\documentclass[12pt]{amsart}
\usepackage{amscd,amssymb}
\usepackage[pagewise]{lineno}%\linenumbers

\usepackage{graphicx,enumerate}
\usepackage{color}
\usepackage[arrow,matrix,arrow,ps,color,line,curve,frame]{xy}
\usepackage{pgf, tikz}
\usepackage{url}
\usepackage{enumerate}

\usepackage[colorinlistoftodos,bordercolor=red,backgroundcolor=red!20,linecolor=red,textsize=scriptsize]{todonotes}

\let\oldlabel=\label

\def\prellabel{\marginparsep=1em
    \def\label##1{\oldlabel{##1}\ifmmode\else\ifinner\else
         \marginpar{{\footnotesize\ \\ \tt
                    ##1}}\fi\fi}}
                    
\def\height{\operatorname{ht}}
\def\D{\operatorname{D}}

\def\Im{\operatorname{Im}}

\def\inte{\operatorname{int}}
\def\Hilb{\operatorname{Hilb}}
\def\vol{\operatorname{vol}}
\def\Aff{\operatorname{Aff}}

\def\unim{\operatorname{Unim}}

\def\Lat{\mathsf{L}}

\def\para{\mathsf{par}}
\def\Lpara{\mathsf{Lpar}}

\def\e{\mathbf{e}}

\def\np{{\operatorname{NPol}}}
\def\pol{{\operatorname{Pol}}}

\def\cones{\operatorname{Cones}}

\def\RR{\mathbb R}

\def\ZZ{\mathbb Z}
\def\NN{\mathbb N}
\def\FF{\mathbb F}

\let\phi=\varphi
\let\epsilon=\varepsilon

\advance\headheight1.15pt \marginparwidth=2.5cm

\newtheorem{lemma}{Lemma}[section]  \newtheorem{theorem}[lemma]{Theorem}
\newtheorem{proposition}[lemma]{Proposition} \newtheorem{conjecture}[lemma]{Conjecture}

\theoremstyle{definition}  \newtheorem{remark}[lemma]{Remark}
\newtheorem{example}[lemma]{Example} \newtheorem{question}[lemma]{Question}

\begin{document}

\title[The poset of rational cones]{The poset of rational cones}

\author{Joseph Gubeladze}
\author{Mateusz Micha$\l$ek}

\address{Department of Mathematics\\
         San Francisco State University\\
         1600 Holloway Ave.\\
         San Francisco, CA 94132, USA}
\email{soso@sfsu.edu}

\address{
Freie Universit\"at\\
 Arnimallee 3\\
 14195 Berlin, Germany\newline
Polish Academy of Sciences\\
         ul. \'Sniadeckich 8\\
         00-956 Warsaw\\
         Poland\newline
MPI MiS, Inselstr. 22,  Leipzig, Germany}
\email{wajcha2@poczta.onet.pl}

\thanks{Gubeladze was supported by NSF grant  DMS 1301487}

\thanks{Michalek was supported by Polish National Science Center grant
2013/08/A/ST1/00804 and Foundation for Polish Science (FNP)}


\begin{abstract} 
We introduce a natural partial order on the set $\cones(d)$ of rational cones in $\RR^d$. The poset of normal polytopes, studied in \cite{Jumps}, embeds into $\cones(d)$ via the homogenization map. The order in $\cones(d)$ is conjecturally the inclusion order. We prove this for $d=3$ and show a stronger version of the connectivity of $\cones(d)$ for all $d$. Topological aspects of the conjecture are also discussed.  
\end{abstract}

\maketitle

\section{Introduction}\label{INTRO}
Rational cones in $\RR^d$ are important objects in toric algebraic geometry, combinatorial commutative algebra, geometric combinatorics, integer programming \cite{Beck,Kripo,Cox,Miller-Sturmfels,Schrijver}. The interaction of these convex objects with the integer lattice $\mathbb Z^d$ is governed by their \emph{Hilbert bases} -- the finite sets of indecomposable elements, notoriously difficult to characterize. General results on Hilbert bases are available only in low dimensions, e.g., \cite{Aguzzoli,Bouvier,SebHILB}. In higher dimensions there are mostly counterexamples to conjectures, e.g. \cite{caratex2,Unico,Caratex}. In this paper we introduce a partial order on the set of rational cones in $\RR^d$. It is defined in terms of the additive generation of the sets of lattice points in cones. The resulting poset $\cones(d)$ is a structure in its own right, which captures a global picture of the interaction of $\ZZ^d$ with all cones at once. The poset $\np(d-1)$ of \emph{normal polytopes} in $\RR^{d-1}$, introduced in \cite{Jumps}, monotonically embeds into $\cones(d)$ via the \emph{homogenization} map. But the former poset is much more difficult to analyze than $\cones(d)$. In fact, there are maximal and nontrivial minimal normal polytopes; at present even the presence of isolated normal polytopes is not excluded \cite{Jumps}. On the other extreme, we conjecture that the order in $\cones(d)$ is just the inclusion order (Conjecture \ref{conjecture-additive}).  We prove the 3-dimensional case of the conjecture (Theorem \ref{dim-3}) and a stronger version of the connectivity of $\cones(d)$ for all $d$: any two cones can be connected by a sequence of $O(d)$ many elementary extensions/descents, or $O(d^2)$ many such moves if working with the full-dimensional cones (Theorem \ref{diameter}). In Section \ref{Topology} we consider topological consequences of Conjecture \ref{conjecture-additive}.

\medskip\noindent\emph{Acknowledgment.} We are grateful to Andreas Paffenholz for carrying out computational experiments, discussed in Section \ref{Cone-conjecture}.

\subsection{Cones}\label{Basics}
We consider the real vector space $\RR^d$, consisting of $d$-columns, together with the integer  lattice $\ZZ^d$. The standard basis vectors will be denoted by $\e_1,\ldots,\e_d$, the set of non-negative reals will be denoted by $\RR_+$, and that of non-negative integers will be denoed by $\ZZ_+$.

For a subset $X\subset\RR^d$ its \emph{conical hull}, i.e., the set of nonnegative linear combinations of elements of $X$, is denoted by $\RR_+X$. The linear span of $X$ will be denoted by $\RR X$. We also put $\Lat(X)=X\cap\ZZ^d$.

By a \emph{cone} $C$ we always mean a pointed, rational, polydedral cone, i.e., $C=\RR_+x_1+\cdots+\RR_+x_n$ for some $x_1,\ldots,x_n\in\ZZ^d$ and there is no nonzero element $x\in C$ with $-x\in C$. Let $C\subset\RR^d$ be a nonzero cone. Then there exists an affine hyperplane $H$, meeting $C$ \emph{transversally}, i.e., such that $C\cap H$ is a polytope of dimension $\dim(C)-1$ \cite[Proposition 1.21]{Kripo}. The first nonzero lattice point on each 1-dimensional face of $C$ is called an \emph{extremal generator of} $C$. The additive submonoid $\Lat(C)\subset\ZZ^d$ has the smallest generating set, consisting of indecomposable elements. It is called the \emph{Hilbert basis} of $C$, denoted by $\Hilb(C)$. The extremal generators of $C$ belong to $\Hilb(C)$.

A $d$-cone $C\subset\RR^d$ has a unique minimal representation as an intersection of closed half-spaces $C=\bigcap_{j=1}^n H_j^+$. The boundary hyperplanes $H_j\subset H_j^+$ intersect $C$ in its  \emph{facets}, i.e., the codimension 1 proper faces of $C$. Further, for each facet $F\subset C$ there exists a unique linear function $\height_F:\RR^d\rightarrow\RR$ which vanishes on $F$, is non-negative on $C$, and satisfies $\height_F(\ZZ^d)=\ZZ$.

A pair of cones $(C,D)$  is a \emph{unimodular extension} of cones if $C$ is a facet of $D$, the latter has exactly one extremal generator $v$ not in $C$, and $\Lat(D)=\Lat(C)+\ZZ_+v$. 

A cone $C\subset\RR^d$ is called \emph{unimodular} if $\Hilb(C)$ is a part of a basis of $\ZZ^d$. 

If the extremal generators of a cone $C$ are linearly independent, then $C$ is said to be \emph{simplicial}.

For elements $u_1,\ldots,u_d\in\RR^d$ the matrix, whose $i$-th column is $u_i$, will be denoted by $[u_1|\ldots|u_d]$. Assume $u_1,\ldots,u_d$ are linearly independent and $C=\RR_+u_1+\cdots+\RR_+u_d$. Then we put:
\begin{align*}
&\para(u_1,\ldots,u_d)=\{\lambda_1u_1+\cdots+\lambda_du_d\ |\ 0\le\lambda_1,\ldots,\lambda_d<1\},\\
&\Lpara(u_1,\ldots,u_d)=\Lat(\para(u_1,\ldots,u_d))\setminus\{0\},\\
&\vol(u_1,\ldots,u_d)=\vol\big(\para(u_1,\ldots,u_d)\big)=|\det[u_1|\ldots| u_d]|,\\
&\mu(C)=\vol(u_1,\ldots,u_d)\ \text{if the}\ u_i\ \text{are \emph{primitive,} i.e., have coprime components}.
\end{align*}

A triangulation of a cone $C$ into simplicial cones is called \emph{unimodular} if the cones in the triangulation are unimodular, and it is called \emph{Hilbert} if the set of extremal generators of the involved cones equals $\Hilb(C)$. 

\begin{proposition}\label{summary}
\begin{enumerate}[\rm(a)]
\item
Let $C\subset\RR^d$ be a nonzero cone and $v\in\Lat(C)$ be a nonzero element in a 1-face of $C$. Then $\Lat(C)+\ZZ v=\Lat(C_0)+\ZZ v\cong\Lat(C_0)\times\ZZ v$ for some cone $C_0\subset\RR^d$ with $v\notin C_0$. 
\item
Let $C\subset\RR^d$ be a nonzero cone and $w\in\Lat(C)$ be an element in the relative interior of $C$. Then $\Lat(C)+\ZZ w=\Lat(\RR C)$.
\item Every nonzero cone has a unimodular triangulation.
\item For every 2-cone $C$ its only Hilbert triangulation is unimodular.
\item Every 3-dimensional cone has a unimodular Hilbert triangulation.
\end{enumerate}
\end{proposition}

The parts (a,b,c,d) are standard results on cones and all five parts are proved, for instance, in \cite[Chapter 2]{Kripo}. The part (e) is originally due to Seb\H o \cite{SebHILB} (whose argument is reproduced in \cite[Theorem 2.78]{Kripo}). It was later rediscovered in the context of toric geometry in \cite{Aguzzoli,Bouvier}, with important refinements. The existence of unimodular Hilbert triangulations fails already in dimension 4 \cite{Bouvier}.

For a poset $(\Pi,<)$, the geometric realization of its \emph{order (simplicial) complex} will be called the \emph{geometric realization} of $\Pi$ and denoted by $|\Pi|$. For generalities on poset topology we refer the reader to \cite{Wachs}, with the caution that our posets are mostly infinite. But the `finite vs. infinite' dichotomy never plays a role in our treatment. Section 1 in Quillen's foundational work on higher algebraic $K$-theory \cite{Quillen} remains an indispensable source for homotopy studies of general posets (in fact, general categories).  

\subsection{The poset of normal polytopes}\label{np(d)} A lattice polytope $P\subset\RR^d$ (i.e., a convex polytope with vertices in $\ZZ^d$) is \emph{normal} if for every $c\in\NN$ and every element $x\in\Lat(cP)$ there exist $x_1,\dots,x_c\in\Lat(P)$, such that $x=x_1+\cdots+x_c$. 

The order in the poset $\np(d)$ of normal polytopes in $\RR^d$, studied in \cite{Jumps}, is generated by the following elementary relations:
$P<Q\ \text{if}\ P\subset Q\ \text{and}\ \#\Lat(Q)=\#\Lat(P)+1$. 

The poset $\np(d)$ is known to have (nontrivial) minimal and maximal elements in dimensions $\ge4$.

The \emph{homogenization} map $P\mapsto C(P):=\RR_+(P\times\{1\})\subset\RR^d$ embeds the set of lattice polytopes $P\subset\RR^{d-1}$ into that of cones $C\subset\RR^d$. Moreover, a lattice polytope $P$ is normal if and only if $\Hilb(C(P))=\{(x,1)\ |\ x\in\Lat(P)\}$.

For a lattice $d$-polytope $P\subset\RR^d$ and a facet $F\subset P$ there exists a unique affine map $\height_F:\RR^d\to\RR$ with $\height_F(P)\subset\RR_+$ and $\height_F(\ZZ^d)=\ZZ$. We have the following compatibility between the facet-height functions: the two maps $\height_F(\cdot),\height_{C(F)}(\cdot,1):\RR^d\to\RR$ are same.

\section{The poset $\cones(d)$}\label{Cones}

\subsection{Elementary extensions} For a natural number $d$ we denote by $\cones(d)$ the set of cones $C\subset\RR^d$, made into a poset as follows: $C< D$ if and only if there exists a sequence of cones of the form
\begin{equation}\label{Cadd}
\begin{aligned} &C=C_0\subset\ldots\subset C_{n-1}\subset C_n=D,\\ &\Lat(C_i)=\Lat(C_{i-1})+\ZZ_+x\
\ \text{for some}\ \ x\in C_i\setminus C_{i-1},\quad i=1,\ldots,n. \end{aligned} \end{equation}

When $n=1$ we call $C\subset D$ an \emph{elementary extension}, or \emph{elementary descent} if read backwards. Here is an alternative characterization:

\begin{lemma}\label{corner-triangulation}
Let $C\subset\RR^d$ be a nonzero cone and $v\in\ZZ^d$ be a primitive vector with $\pm v\notin C$. Assume $H\subset\RR^d\setminus\{0\}$ is an affine hyperplane, meeting the cone $D=C+\RR_+v$ transversally. Put $v'=\RR_+v\cap H$. Then $C\subset D$ is an elementary extension in $\cones(d)$ if and only if there exist unimodular cones $U_1,\ldots,U_n\subset D$, satisfying the conditions:
\begin{enumerate}[\rm(i)]
\item $v\in U_i$, $i=1,\ldots,n$,
\item $D=C\ \bigcup\ \left(\bigcup_{i=1}^n U_i\right)$,
\item $\big\{\RR_+\big((U_i\cap H)-v'\big)\big\}_{i=1}^n$ is a triangulation of the cone $\RR_+\big((D\cap H)-v'\big)$.
\end{enumerate}
\end{lemma}

\begin{proof}
The `if' part is obvious. For the `only if' part we use Proposition \ref{summary}(a,c) to fix a representation
$\Lat(C)+\ZZ v=\Lat(D)+\ZZ v=\Lat(C_0)+\ZZ v=\Lat(C_0)\times\ZZ v$ and a unimodular triangulation $C_0=\bigcup_{i=1}^nD_i$. Let $X\subset C$ be a finite subset, which maps bijectively to $\bigcup_{i=1}^n\Hilb(D_k)$ under the projection $C\to C_0$, induced by $v\mapsto0$. Let $X_i$ be the preimage of $\Hilb(D_i)$ in $X$. Then the cones $U_i=\RR_+X_i+\RR_+v$ satisfy (i--iii).
\end{proof}

\subsection{Height 1 and Hilbert basis extensions}\label{Height 1} There is a ubiquity of two types of elementary extensions in $\cones(d)$, making this poset essentially different from the more rarefied $\np(d)$. 

Let $C\subset\RR^d$ be a $d$-cone and $v\in\ZZ^d$ with $\pm v\notin C$. Denote by $\FF^+(v)$ the set of facets of $C$, visible from $v$, i.e., $\height_F(v)<0$ for every $F\in\FF^+(v)$. Consider the visible part of the boundary $\partial C$, i.e., $C^+(v)=\bigcup_{\FF^+(v)}F$.
Put $D=C+\RR_+v$. There is a sequence of rational numbers $0<\lambda_1<\lambda_2<\ldots$ with $\lambda_1=\frac1{\max(-\height_F(v)\ :\ F\in\FF^+(v))}$ and $\lim_{k\to\infty}\lambda_k=\infty$, satisfying the conditions:
\begin{align*} 
\Lat(D\setminus
C)=\bigcup_{k=1}^\infty\Lat(\lambda_kv+C^+(v))\ \text{and}\ \ \Lat(\lambda_kv+C^+(v))\not=\emptyset,\quad k=1,2,\ldots\\
\end{align*}
The equality $\lambda_1=1$ equivalent to the condition $\height_F(v)=-1$ for all $F\in\FF^+(v)$. In this case we say that $D$ is a \emph{height 1 extension} of $C$. All height 1 extensions are elementary extensions of cones but the converse is not true \cite[Theorem 4.3]{Jumps}.%, as the last column in the table in Section \ref{np(d)} shows, in view of the compatibility between a facet-height function of a polytope and that of its homogenization.

The second class of elementary extensions in $\cones(d)$ are the extensions of type $C\subset D$, where $\RR_+(\Hilb(D)\setminus\{v\})\subset C$ for an extremal generator $v\in D$. We call this class the \emph{Hilbert basis extensions} (or \emph{descents}).

As an application of the two types of extensions, we have
\begin{lemma}\label{no-extremal-elements}
For every natural number $d\ge2$:
\begin{enumerate}[\rm(a)]
\item
Every elementary extensions of cones $0\not=C<D$ there exists a cone $E$, such that $C<E<D$;
\item $\cones(d)$ has neither maximal nor minimal elements, except the only minimal element 0.
\end{enumerate}
\end{lemma}

\noindent\emph{Remark.} We do not know whether $0$ is the \emph{smallest} element of $\cones(d)$. If $0$ was the smallest element, then the geometric realization of $\cones(d)$ would be contractible; see Section \ref{Topology} for topological aspects of $\cones(d)$.

\begin{proof} (a) The general case easily reduces to the full-dimensional case and then the claim follows from the observation that there is always a height 1 extension $C\subset E$ with $E\subsetneq D$. In fact, if $\{v\}=\Hilb(D)\setminus C$, then we can take $E=C+\RR_+w$ where $w\in\Lat(\lambda_1v+C^+(v))$  with $w\not=v$ (notation as above). Obviously, $E\subset D$ is an elementary extension.

For (b) one applies appropriate height 1 extensions to show that there are no maximal elements and Hilbert basis descents to show that there are no minimal elements in $\cones(d)\setminus\{0\}$.
\end{proof}

We formally include the extensions of type $0\subset C$, $\dim C=1$, in both classes of elementary extensions, discussed above.

\begin{question}\label{unimodular?}
Do either the height 1 or Hilbert basis extensions generate the same poset $\cones(d)$?
\end{question}

\subsection{Distinguished subposets}\label{Subposets}
The subposet of $\cones(d)$, consisting of the cones in  $\big(\RR^{d-1}\times\RR_{>0}\big)\cup\{0\}$, will be denoted by $\cones^+(d)$.  The homogenization embedding $\np(d-1)\to\cones^+(d)$ is a monotonic map. However, the order in $\np(d-1)$ is weaker than the one induced from $\cones(d)$:

\begin{example}\label{C12}
In \cite[Example 4.8]{Jumps}  we have the polytope $P\in\np(3)$ with vertices: $(0,0,2),(0,0,1),(0,1,3),(1,0,0),(2,1,2),(1,2,1)$. The polytope has two more lattice points: $(1,1,2),(1,1,1).$ Removing either the first \emph{or} the second vertex and taking the convex hull of the other lattice points in $P$ yields a nonnormal polytope. However, the convex hull $Q$ of the lattice points in $P$ with the exception of the first two vertices is normal. We have $Q\not< P$ in $\np(3)$. Yet, using \textsf{Polymake} \cite{Poly}, one quickly finds four Hilbert basis descents (requiring additional Hilbert basis elements at height two) $C(P)>C_1>C_2>C_3>C(Q)$.

\end{example}

For every integer $h>0$ we consider the poset $\cones^{(h)}(d)$ of cones in $\cones^+(d)$, satisfying $\Hilb(C)\subset\RR^{d-1}\times[0,h]$ and ordered as in (\ref{Cadd}) under the additional requirement that the intermediate cones $C_i$ are also from $\cones^{(h)}(d)$.

%Let $\pol(d-1)$ denote the set of all lattice polytopes in $\RR^{d-1}$. It can be viewed as a subset of $\cones^+(d)$ via the homogenization map $\pol(d-1)\to\cones^+(d)$, which is a monotonic map of posets.

\begin{lemma}\label{Inclusions}
For every natural $d\ge1$:
\begin{enumerate}[\rm(a)]
\item $\cones^{(1)}(d)\setminus\{0\}=\np(d-1)$;
\item $\cones^{(1)}(d)\subset\cones^{(2)}(d)\subset\ldots$ and $\bigcup_{h=0}^\infty\cones^{(h)}(d)=\cones^+(d)$;
\item $\pol(d-1)\subset\cones^{(d-2)}(d)$, assuming $d\ge3$.
\end{enumerate}
\end{lemma}

\noindent The inclusions above are those of sets, which may not represent subposets.

The parts (a,b) are obvious; (c) is proved, for instance, in \cite[Theorem 2.52]{Kripo}.

\subsection{The cone conjecture}\label{Cone-conjecture} The following conjecture is the maximal possible strengthening of the absence of extremal elements in $\cones(d)$:

\begin{conjecture}\label{conjecture-additive} For every $d$, the order in $\cones(d)$ is the inclusion order. 
\end{conjecture}

The case $d=1$ is obvious.

When $d=2$, the general case reduces to a pair of cones $C\subset D$ in $\RR^2$, with $\dim D=2$ and $C$ a facet of $D$. Assume $\{v_1,\ldots,v_n\}=\Hilb(D)$ and $v_1\in C$. Then, by Proposition \ref{summary}(d), we have the following height 1 extensions:
\begin{align*}
C<C+\RR_+v_2<\ldots<C+\RR_+v_2+\cdots+\RR_+v_n=D.
\end{align*} 

In Sections \ref{3D} we give a proof for $d=3$. 

In dimension 4 we have the following computational evidence. 

Assume $C\subset\RR^d$ is a cone and $v\in\ZZ^d$ with $\pm v\notin C$. We use the notation in Section \ref{Height 1}. In particular, $D=C+\RR_+v$. One introduces the \emph{bottom-up} procedure for constructing an ascending sequence of height 1 extensions, starting with the cone $C$, as follows: one chooses a shortest vector $v_1\in\Lat(\lambda_1v+C^+(v))$, repeats the step for the pair $C_1\subset D$ where $C_1=C+\RR_+v_1$, and iterates the process. The height 1 extensions we obtain this way tend to widen the cone as much as possible at each step, as measured by the increments of the Euclidean $(d-1)$-volume of the cross sections with a pre-chosen affine hyperplane, transversally meeting the cone $D$. 

A complementary approach employs Hilbert basis descents. The corresponding \emph{top-down} procedure finds a sequence $D=D_0>D_1>\ldots$ of Hilbert basis descents of the form $D_{i+1}=C+\RR_+(\Hilb(D_i)\setminus\{v_i\})$, at each step discarding a shortest extremal generator $v_i\in D_i\setminus C$.

Andreas Paffenholz implemented the bottom-up and top-down procedures in $\RR^4$. The computational evidence, based on many randomly generated cones $C$ and vectors $v$, supports the expectation that there are no non-terminating processes of either type, with the tendency of the bottom-up process to last longer than the top-down one. %It is possible that already height 1 and Hilbert basis extensions separately generate the inclusion order $\cones(d)$.

\section{Cones in $\RR^3$}\label{3D}

\begin{lemma}\label{3-inequality}
Let $u,v,w\in\RR^3$ be linearly independent vectors and $x,y\in\para(u,v,w)$. Then $\vol(u,x,y)<\vol(u,v,w)$.
\end{lemma}

\begin{proof}
We can assume $(u,v,w)=(\e_1,\e_2,\e_3)$. Let $x=(x_1,x_2,x_3)$ and $y=(y_1,y_2,y_3)$. Then
\begin{align*}
\vol(\e_1,x,y)=\left|\det
\begin{pmatrix}
1&0&0\\
x_1&x_2&x_3\\
y_1&y_2&y_3\\
\end{pmatrix}
\right|
=|&x_2y_3-x_3y_2|\le\\
&\max\big(|x_2y_3|,|x_3y_2|\big)<1.
\end{align*}
\end{proof}

\begin{theorem}\label{dim-3}
The order in $\cones(d)$ is the inclusion order for $d=3$.
\end{theorem}

\begin{proof}
We first prove the following basic case: for any simplicial 3-cone $D\subset\RR^3$ and any facet $C\subset D$ we have $C< D$. This will be done by induction on $\mu(D)$ (defined in the introduction).

The case $\mu(C)=\mu(D)$ is obvious because $D$ is a unimodular extension of $C$. So we can assume $\mu(C)<\mu(D)$, which is equivalent to $\Lpara(D)\not\subset C$. 

Let $v_0,v_1,w$ be the extremal generators of $D$ with $v_0,v_1\in C$. Denote by  $v_0,v_1,v_2,\ldots,v_k$ ($k\ge2)$ the extremal generators of the cone 
$$
E=C+\RR_+\Lpara(D)\subset\RR^3.
$$
We assume that the enumeration is done in the cyclic order, i.e., the cones 
\begin{align*}
C_i=\RR_+v_{i-1}+\RR_+v_i\subset\RR^3,\quad i=1,\ldots,k,k+1\mod(k+1)
\end{align*}
are the facets of $E$. (Here, $C=C_{1}$.) 

Because of the containment  $\Hilb(D)\setminus\{w\}\subseteq E$, we have $E<D$ in the poset $\cones(3)$. Further, the cone $E$ is triangulated by the cones 
$$
D_i=\RR_+v_0+\RR_+v_i+\RR_+v_{i+1},\quad i=1,\ldots,k-1.
$$
By Lemma \ref{3-inequality}, we have the inequalities
$$
\mu(D_i)<\mu(D),\quad i=2,\ldots,k.
$$
Then, by the induction hypothesis, we have $C<D_1$ and
$$
(D_{i-1}\cap D_i)<D_i,\quad i=2,\ldots,k-1.
$$
By concatenating, we obtain the following chain in $\cones(3)$:
$$
C<D_1<D_1\cup D_2<\ldots<D_1\cup D_2\cup\ldots\cup D_{k-1}=E<D.
$$
This completes the proof of the basic case.

\medskip The general case easily reduces to the case of a pair of $3$-cones $C\subsetneq D$ with $D=C+\RR_+v$, to which we apply induction on the number of facets of $C$ visible from $v$. When this number is $1$, the inequality $C<D$ results from the basic case. When the number of the visible facets is $k\ge2$ then there is an intermediate cone $C\subsetneq B\subsetneq D$, satisfying the conditions:
\begin{enumerate}[\rm$\centerdot$]
\item $B=C+\RR_+w$ for some $w$;
\item $B$ has only one facet visible from $v$;
\item there are exactly $k-1$ facets of $C$, visible from $w$.
\end{enumerate}
In fact, if $C=\bigcap_{j=1}^lH_j^+$ is the irreducible representation, where the indexing is in the circular order and $H_1\cap C,\ldots,H_k\cap C\subset C$ are the facets visible from $v$, then one can choose 
$$
B=\big(\bigcap_{j=k}^lH_j^+\big)\bigcap D.
$$
We are done because, by the induction hypothesis, $C<B<D$.
\end{proof} 

\section{Diameter}\label{Diameter}

By the \emph{diameter} of a subposet $X\subset\cones(d)$, denoted $\D(X)$, we mean the supremum of the lengths of the shortest sequences $C_0C_1\ldots C_n$ within $X$, connecting any two elements $C_0,C_n$ of $X$, where every two consecutive cones form an elementary extension or descent.

Consider the following subposets of $\cones(d)$:
\begin{enumerate}[\rm(i)]
\item $\cones(d)^{\text o}$, consisting of the $d$-cones in $\RR^d$ (all quantum jumps in $\np(d-1)$ live here);
\item $\unim(d)$, consisting of the unimodular cones in $\RR^d$;
\item $\unim(d)^{\text o}$ consisting of the unimodular $d$-cones  in $\RR^d$.
\end{enumerate}
 
The next theorem implies that $\cones(d)$ and $\cones(d)^{\text o}$ are  both connected.

\begin{theorem}\label{diameter}
We have:
\begin{enumerate}[\rm(a)]
\item $\D(\unim(d))=2d$ for every $d\in\NN$,
\item $\D(\unim(d)^{\emph{\text o}})=O(d^2)$,
\item $\D(\cones(d))=O(d)$,
\item $\D(\cones(d)^{\emph{\text o}})=O(d^2)$.
\end{enumerate}
\end{theorem}

\begin{proof} (a) Any unimodular cone can be reached from any other unimodular cone by first removing the Hilbert basis elements of the latter one by one and then adding those of the former, also one at a time. 

For the pairs of unimodular $d$-cones of type $C$ and $-C$ there is no shorter connecting path. One should remark that this is not true for all pairs of unimodular $d$-cones whose intersection is $0$; an example when $d=2$ is
$$
\RR_+\e_1+\RR_+(\e_1+\e_2)<\RR_+(-\e_1+\e_2)+\RR_+\e_1>
\RR_+(-\e_1+\e_2)+\RR_+\e_2.
$$

\medskip\noindent(b) Let $C=\sum_{i=1}^d\ZZ_+v_i$ and $D=\sum_{i=1}^d\ZZ_+w_i$ for two bases $\{v_1,\ldots,v_d\}$ and $\{w_1,\ldots,w_d\}$ of $\ZZ^d$. Put $A=[v_1|\ldots|v_d]$ and $B=[w_1|\ldots|w_d]$. After renumbering of the basis elements, we can assume $\det(A)=\det(B)=1$. The special linear group $SL_d(\ZZ)$ is generated by the elementary matrices $e_{ij}^a$, i.e., the matrices with $1$-s on the main diagonal, at most one non-zero off-diagonal entry $a$ in the $ij$-spot, and $0$-s elsewhere.
Using the equalities $(e_{ij}^a)^{-1}=e_{ij}^{-a}$, there is a representation of the form $Ae_{i_1j_1}^{a_1}\cdots e_{i_kj_k}^{a_k}=B$, where $a_1,\ldots a_k\in\ZZ$. By \cite{Carter}, one can choose $k\le36+\frac12(3d^2-d)$. 
Consider the sequence of unimodular cones:
$$
C_t\ =\ \text{the cone spanned by the columns of}\ 
Ae_{i_1j_1}^{a_1}\cdots e_{i_tj_t}^{a_t},\quad 0\le t\le k.
$$
(In particular, $C_0=C$). Since the multiplications  by elementary matrices from the right corresponds to the elementary column transformations, for every $1\le t\le k$ the inequality $a_t>0$ yields the elementary extension $C_t<C_{t-1}$ and the inequality $a_t<0$ yields the elementary descent $C_t>C_{t-1}$. 

\medskip\noindent(c,d) For $d\le1$ there are connecting paths of length $\le2$. So we assume $d\ge2$.

Pick $C\in\cones(d)$. By taking unimodular extensions as needed, we can assume $\dim C=d$.  We need at most $2d-1$ unimodular extensions to reach the full dimensional case. Consequently, the parts (c) and (d) follow from the parts (a) and (b), respectively, once we show that a unimodular $d$-cone can be reached from $C$ in at most $d-1$ elementary extensions/descents.

Pick arbitrarily a facet $F\subset C$ and two elements $y\in\Lat(C\setminus F)$, satisfying $\height_F(y)=1$, and $x\in\Lat(\inte(F))$, where $\inte(F)$ is the relative interior of $F$. Consider the sequence of cones 
\begin{align*}
C_k=F+\RR_+(y-kx),\qquad k=0,1,\ldots
\end{align*}

We claim that $C\subset C_k$ for all sufficiently large $k$.

Indeed, consider any extremal generator $v$ of $C$. We have $v=\height_F(v)y+v'$ for some $v\rq{}\in\ZZ^d$ with $H_F(v')=0$. By Lemma \ref{summary}(b), $\Lat(F)+\ZZ x=\Lat(\RR F)$. Hence $v'=-sx+z$ for some $z\in\Lat(F)$ and an integer $s\ge0$. Consequently, 
$$
v=\height_F(v)\left(y-\left\lceil \frac{s}{\height_F(v)}\right\rceil x\right)+\height_F(v)\left(1-\left\{\frac{s}{\height_F(v)}\right\}\right)x+z\in C_{\left\lceil\frac{s}{\height_F(v)}\right\rceil}.
$$

Pick $k\gg0$ with $C\subset C_k$. Since $C_k$ is a unimodular extension of $F$, we have the elementary extension $C<C_k$ in $\cones(d)$.

Keeping $\RR_+(y-kx)$ as a 1-face, we may, inductively on dimension, transform $F$ to a unimodular $d-1$-cone using only elementary extensions and descents: one uses the fact that unimodular extensions of cones respect elementary extensions in the previous dimension. In the end, starting from $C$, we have reached a unimodular $d$-cone (in at most $d-1$ steps).
\end{proof}

\begin{remark}\label{uniform} (a) In the proofs of Theorem \ref{diameter}(a,c) one does not need to descent from unimodular cones all the way to $0$. The latter, not being in $\np(d-1)$, may not be desirable. It is enough to descent to $1$-dimensional cones and the same argument as in the proof of Theorem \ref{diameter}(b) shows that for any pair of 1-cones in $\RR^d$ there is an upper bound on the number of connecting elementary extensions/descents: one finds such extensions within the linear span of the pair of 1-cones. By avoiding 0 the diameter goes up by a constant, independent of $d$.

(b) The proof of Theorem \ref{diameter} does \emph{not} imply that $\D(\cones^+(d))<\infty$.
\end{remark}

\section{The space of cones}\label{Topology}

Conjecture \ref{conjecture-additive} has strong consequences for the geometric realization of $\cones(d)$:

\begin{theorem}\label{nice-properties} Assume Conjecture \ref{conjecture-additive} holds for a natural number $d$. Then:
\begin{enumerate}[{\rm(a)}]
\item The spaces $|\cones(d)|$,  $|\cones^+(d)|$, and  $|\cones^+(d)\setminus\{0\}|$ are contractible;
%\item The induced order on $\pol(d-1)$ is the inclusion order;
\item $|\cones(d)\setminus\{0\}|$ is a filtered union of spaces, each containing a $(d-1)$-sphere as a strong deformation retract. 
\end{enumerate}
\end{theorem}

\begin{proof} (a) The spaces $|\cones(d)|$ and  $|\cones^+(d)|$ are contractible because $0$ is the smallest element of $\cones(d)$ and $\cones^+(d)$. The poset $\cones^+(d)\setminus\{0\}$ is \emph{filtering,} i.e., every finite subset has an upper bound. But the geometric realization of a filtering poset is contractible (\cite[Section 1]{Quillen}). 

%\medskip\noindent The part (b) is obvious.

\medskip\noindent (b) Let $S^{d-1}$ be the unit $(d-1)$-sphere in $\RR^d$, centered at the origin. Then we can think of the poset of $\cones(d)\setminus\{0\}$ as the poset of intersections $C\cap S^{d-1}$, $C\in\cones(d)$, ordered by inclusion. Abusing terminology, these intersections will be also called polytopes.

For two polytopal subdivisions $\Pi_1$ and $\Pi_2$ of $S^{d-1}$ and a polytope $P\subset S^{d-1}$ we write (i) $\Pi_1\prec\Pi_2$ if $\Pi_2$ is a subdivision of $\Pi_1$ and (ii) $P\prec\Pi_1$ if $P$ is subdivided by  polytopes in $\Pi_1$.

Fix a system of polytopal subdivisions $\{\Pi_i\}_{i=1}^\infty$ of $S^{d-1}$, satisfying the conditions: $\Pi_i\prec\Pi_{i+1}$ for all $i$ and every polytope $P\subset S^{d-1}$ admits $i$ with $P\prec\Pi_i$.

For every index $i$, the simplicial complex $|\Pi_i|$ is a barycentric subdivision of $\Pi_i$. In particular, $|\Pi_i|\cong S^{d-1}$. 

Consider the following posets:
\begin{enumerate}[\rm$\centerdot$]
\item $\check\Pi_i=\{P\in\cones(d)\setminus\{0\}\ |\ P\prec \Pi_i\}$, made into a poset by adding to the inclusion order in $\Pi_i$ the new relations $Q<P$ whenever $P\in\check\Pi_i\setminus\Pi_i$, $Q\in\Pi_i$, $Q\subset P$; in particular, two different polytopes $P,P\rq{}\in\check\Pi_i\setminus\Pi_i$ are not comparable;
\item The subposet $\bar\Pi_i=\{P\in\cones(d)\setminus\{0\}\ |\ P\prec \Pi_i\}\subset\cones(d)$; it has more relations then the poset $\check\Pi_i$, supported by the same set of polytopes: for $P,P\rq{}\in\check\Pi_i\setminus\Pi_i$ one has $P<P\rq{}$ whenever $P\subset P\rq{}$;
\item
The subposets $\Pi_i(P)=\{Q\ |\ Q\in\Pi_i,\ Q\subset P\}\cup\{P\}\subset\cones(d)$ for $P\prec\Pi_i$.
\end{enumerate}

The (geometric) simplicial complex $|\check\Pi_i|$ is obtained from $|\Pi_i|$ by changing the contractible subcomplexes $|\Pi_i(P)|$ to pyramids over them. Any two of these pyramids either do not meet outside $|\Pi_i|$ or overlap along a pyramid from the same family. In particular, the subspace $|\Pi_i|\subset|\check\Pi_i|$ is a strong deformation retract. Let $F:|\check \Pi_i|\times[0,1]\to|\check\Pi_i|$ be a corresponding homotopy. 

Consider an extension of $F$ to a homotopy $G:|\bar\Pi_i|\times[0,1]\to|\bar\Pi_i|$, satisfying the condition: for every $t\in[0,1]$ the map $G_t$ is injective on $|\bar\Pi_i|\setminus|\check\Pi_i|$ and is the identity on $|\Pi_i|$. In more detail, for every chain 
$$
P_0<\ldots<P_k<P_{k+1}<\cdots<P_n,\quad P_k\in\Pi_i,\quad P_{k+1}\in\check\Pi_i\setminus\Pi_i,
$$
and every index $k<l\le n$, the $l$-subsimplex $\triangle(P_0,\ldots,P_k,P_l)$ of the $n$-simplex $\triangle(P_0,\ldots,P_n)$  is collapsed by the homotopy $G$ into the $k$-subsimplex $\triangle(P_0,\ldots,P_k)$, while the rest of the $n$-simplex homeomorphically remains invariant. In particular, $G_1(\triangle(P_1,\ldots,P_n))$ is an $n$-disc, attached to $|\Pi_i|$ along the subdisc $\triangle(P_1,\ldots,P_k)$. Then $\Im G_1$ consists of $|\Pi_i|$ and the mentioned finitely many attached discs, any two of which either do not meet outside $|\Pi_i|$ or overlap along a disc from the same family.

The claim now follows because $|\Pi_i|$ is a strong deformation retract of $\Im G_1$.
\end{proof}

\noindent\emph{Remark.} It is very likely that a more elaborate homotopy leads to a deformation retraction of the total space $|\cones(d)\setminus\{0\}|$ to a $(d-1)$-sphere.

\medskip By Lemma \ref{Inclusions}(c), we have the tower of spaces 
$$
|\np(d-1)|=|\cones^{(1)}(d)\setminus\{0\}|\subset|\cones^{(2)}(d)\setminus\{0\}|\subset\ldots,
$$
which, in view of Theorem \ref{nice-properties}, is expected to trivialize in the limit. This observation can lead to an insight into the more difficult space of normal polytopes if the trivialization occurs in a controlled way -- an interesting question in its own right. In more detail, the group $\Aff_{d-1}(\ZZ)$ of affine automorphisms of $\ZZ^{d-1}$ acts compatibly on the whole tower of posets
$$
\cones^{(1)}(d)\setminus\{0\}\subset\cones^{(2)}(d)\setminus\{0\}\subset\cones^{(3)}(d)\setminus\{0\}\subset\ldots$$
via the embedding 
\begin{align*}
\Aff_{d-1}(\ZZ)\to GL_d(\ZZ),\quad(\alpha|\beta)\mapsto
&\begin{pmatrix}
\alpha&\beta\\
0&1
\end{pmatrix},\\
&\qquad\alpha\in GL_{d-1}(\ZZ),\ \beta\in\ZZ^{d-1}.
\end{align*}

As a result, the homology groups of all involved geometric realizations are modules over the group ring $\ZZ[\Aff_{d-1}(\ZZ)]$.

\begin{question}
Are the relative homology groups
$$
H_i\big(|\cones^{(j)}(d)\setminus\{0\}|,|\cones^{(j-1)}(d)\setminus\{0\}|,\ZZ\big)
$$ 
finitely generated $\ZZ[\Aff_{d-1}(\ZZ)]$-modules for all $i$ and $j$?
\end{question}

The positive answer to this question for $i=0$ (and all $j$), would imply that the still elusive isolated elements in $\np(d-1)$ form a highly structured family: for every $j$, only finitely many such isolated elements (up to unimodular equivalence) cease to be isolated when one passes from $\cones^{(j-1)}(d)\setminus\{0\}$ to $\cones^{(j)}(d)\setminus\{0\}$, and all isolated elements are taken out as $j\to\infty$. 

\bibliographystyle{plain}
\bibliography{bibliography} 

\end{document}